\documentclass[10pt]{amsart}

\usepackage{amssymb,amsmath,amsthm}
\usepackage{amsrefs}

\makeatletter
\renewcommand{\pod}[1]{\mathchoice 
  {\allowbreak\if@display\mkern8mu\else\mkern8mu\fi(#1)}
  {\allowbreak\if@display\mkern8mu\else\mkern8mu\fi(#1)}
  {\mkern4mu (#1)}
  {\mkern4mu (#1)}
}

\newcommand{\abs}[1]{
  \ensuremath\lvert#1\rvert}

\makeatother
\newtheorem{lem}{Lemma}
\newtheorem{prop}{Proposition}
\newtheorem{thm}{Theorem}
\theoremstyle{remark}
\newtheorem*{rem*}{Remark}
\numberwithin{equation}{section}

\title{Large Gaps between Primes in Arithmetic Progressions}
\author{Deniz A. Kaptan}
\thanks{Author supported by MTA grant PPD-002/2016}
\address{MTA Alfr\'ed R\'enyi Institute of Mathematics, 13--15 Re\'altanoda u.
1053 Budapest, Hungary}
\email{kaptan@renyi.hu}


\begin{document}

\begin{abstract}
  For $(M,a)=1$, put
\begin{equation*}
  G(X;M,a)=\sup_{p^\prime_n\leq X}(p^\prime_{n+1}-p^\prime_n),
\end{equation*}
where $p^\prime_n$ denotes the $n$-th prime that is congruent to $a\pmod{M}$.
We show that for any positive $C$, provided $X$ is large enough in terms of
$C$, there holds
\begin{equation*}
  G(MX;M,a)\geq(C+o(1))\varphi(M)\frac{\log X\log_2 X\log_4 X}
  {{(\log_3 X)}^2},
\end{equation*}
uniformly for all $M\leq\kappa{(\log X)}^{1/5}$ that satisfy
\begin{equation*}
  \omega(M)\leq \exp\biggl(\frac{\log_2 M\log_4 M}{\log_3 M}\biggr).
\end{equation*}
\end{abstract}

\maketitle

\section{Introduction}
Denote by
\begin{equation}
  G(X)=\sup_{p_n\leq X}(p_{n+1}-p_n)
\end{equation}
the largest gap between consecutive primes up to $X$. The study of how large
$G(X)$ can be has a long history. Westzynthius~\cite{Westzynthius1931Uber}
was the first to show that $G(X)$ can be arbitrarily large compared to the
average gap $(1+o(1))\log X$. Erd\H{o}s~\cite{Erdos1940Difference} and
Rankin~\cite{Rankin1938Difference} showed
\begin{equation}\label{eq:G}
  G(X)\geq(c+o(1))\frac{\log X\log_2 X\log_4 X}{{(\log_3 X)}^2}
\end{equation}
for some positive constant $c$, where $\log_\nu$ denotes the $\nu$-fold
iterated logarithm. Subsequent years saw the constant improved from Rankin's
$1/3$ by various authors---Sch\"onhage~\cite{Schonhage1963Bemerkung},
Rankin~\cite{Rankin1963Difference}, Maier and
Pomerance~\cite{MaierP1990Unusually} among others---with the best constant
$c=2e^\gamma$ due to Pintz~\cite{Pintz1997Very}. After the emergence of the
Maynard-Tao method from the study of the small gaps between primes, the method
was also applied to the large gap problem by Maynard~\cite{Maynard2016Large}
and Ford, Green, Konyagin and Tao~\cite{FordGKT2016Large} independently to show
that~\eqref{eq:G} holds with $c$ arbitrarily large.
Later~\cite{FordGKMT2018Long} the five authors were able to quantify this by
proving that
\begin{equation}
  G(X)\gg\frac{\log X\log_2 X\log_4 X}{\log_3 X}
\end{equation}
holds.

To discuss the corresponding question for primes in an arithmetic progression,
given a modulus $M$ and and a reduced residue class $a\pmod{M}$, put
\begin{equation}
  G(X;M,a)=\sup_{p^\prime_n\leq X}(p^\prime_{n+1}-p^\prime_n),
\end{equation}
where $p^\prime_i$ denotes the $i$-th prime that is congruent to $a\pmod{M}$.
Zaccagnini~\cite{Zaccagnini1992Note} showed that given any positive $C<1$,
uniformly for $M$ satisfying
\begin{equation}
  \omega(M)\leq\exp\biggl(C\log_2 M\frac{\log_4 M}{\log_3 M}\biggr),
\end{equation}
there holds
\begin{equation}
  G(MX;M,a)\geq(e^\gamma+o(1))\varphi(M)\frac{\log X\log_2 X\log_4 X}
  {{(\log_3 X)}^2}.
\end{equation}
The improvements that led to the breakthrough developments in the study of
large gaps between primes naturally lend themselves to the setting of
arithmetic progressions. The present work follows Maynard's
paper~\cite{Maynard2016Large} on large gaps between primes to derive the
analogous result for the case of primes in arithmetic progressions, giving a
lower bound that is uniform in terms of the moduli. Our main result is
\begin{thm}\label{thm:main}
  Let $C>0$ be given. There is an absolute constant $\kappa>0$ such that if
  $X>X_0(C)$ is large enough, we have uniformly for $M\leq
  \kappa{(\log X)}^{1/5}$ satisfying
  \begin{equation}\label{eq:assum}
    \omega(M)\leq \exp\biggl(\frac{\log_2 M\log_4 M}{\log_3 M}\biggr),
  \end{equation}
  and all reduced residues $a\pmod{M}$,  we have
  \begin{equation}\label{eq:mainresult}
    G(X;M,a)\geq(C+o(1))\varphi(M)\frac{\log X\log_2 X\log_4 X}{{(\log_3 X)}^2}.
  \end{equation}
\end{thm}

\section{Setup and the Erd\H{o}s-Rankin construction} Recall that a set of
primes $\mathcal{P}$ is said to \emph{sieve out} an interval $I$ if there is a
choice of residue classes $a_p\pmod{p}$ for each $p\in\mathcal{P}$, such that
for all $n\in I$ there is a $p\in\mathcal{P}$ such that $n\equiv a_p\pmod{P}$.
Our aim is to show, along the lines of the classical Erd\H{o}s-Rankin
construction, that if $M$ is an integer $\leq cx^{1/5}$, then the primes
$p\leq x$, $p\nmid M$ can sieve out the interval $[1,U]$, while taking $U$ as
large as possible with respect to $x$.

We will write $\mathfrak{P}$ to denote all primes and $\mathfrak{P}_x$ to
denote those that don't exceed $x$, and denote by $\mathfrak{P}^{(M)}$ and
$\mathfrak{P}^{(M)}_x$ the same sets with prime divisors of $M$ excluded.  Put
$P_M(x)$ and $P(x)$ for products of primes in $\mathfrak{P}^{(M)}_x$ and
$\mathfrak{P}_x$ respectively.

Suppose that $\mathfrak{P}^{(M)}_x$ can sieve out $[1,U]$, so that
corresponding to each $p \in \mathfrak{P}^{(M)}_x$, there exists a residue
class $a_p\pmod{p}$, such that each number $n=1,\ldots,U$ satisfies $n \equiv
a_p \pmod p$ for some $p$.

By the Chinese Remainder Theorem, in any block of $P_M(x)$, integers, there
is a $U_0$ such that $U_0 \equiv -a_p \pmod{p}$ for each $p \mid P_M(x)$.
Let $j \in [1,U]$, and let $p$ be a prime in $\mathfrak{P}^{(M)}_x$ such that
$j \equiv a_p \pmod{p}$. Let $r$ be such that $Mr \equiv -1 \pmod{P_M(x)}$.
Then for any reduced residue $a\pmod{M}$,
\begin{equation}
  \begin{split}
    M(U_0+ar+j)+a &\equiv M(U_0+j) - a + a\\
    &\equiv M(-a_p + a_p) \equiv 0 \pmod{p},
  \end{split}
\end{equation}
so $M(U_0+ar+j)+a$ is composite provided it is greater than $x$, which is the
case if $\frac{x}{M} \leq U_0$. We would also like to ensure the existence of a
prime in the arithmetic progression preceding our block of composites. That
would follow from the best known result on Linnik's
constant~\cite{Xylouris2011Uber} if $MU_0\geq c_0 M^5$, i.e. $U_0 \geq c_0
M^4$. Here and throughout, $\kappa$ denotes $c_0^{-5}$, where $c_0$ is the
constant for which Linnik's theorem with exponent $5$ is valid. We impose the
condition $M\leq\kappa x^{1/5}$ so that $U_0 \geq c_0 M^4$ is implied by
$\frac{x}{M} \leq U_0$.  So with $U_0 \in [\tfrac{x}{M}, \tfrac{x}{M} +
P_M(x)]$, we find a prime $p_n \leq M(U_0+ar)$ such that $p_{n+1} - p_n \geq
MU$.

Heuristically, since each $a_p\pmod{p}$ removes an element with probability
$1/p$, it is reasonable to expect that the integers we can sieve out using
primes that don't divide $M$ will be less numerous by a factor of $\prod_{p\mid
M}(1-1/p)$ than those we can sieve out using all primes. Accordingly we put
\begin{equation}\label{eq:U}
  U=C_U \frac{\phi(M)}{M} \frac{x \log y}{\log\log x},
\end{equation}
where
\begin{equation}
  y=\exp\left( (1-\varepsilon)\frac{\log x\log\log\log x}{\log\log x} \right).
\end{equation}
and $C_U$ is a constant to be specified later. Now putting
$x=(1-\varepsilon)\log X$ large enough depending only on $\varepsilon$ yields
\begin{gather}
  \begin{aligned}
    G(X;M,a)&\geq G(M(U_0+ar);M,a)\\
    &\geq MU\\
    &=(C_U+o(1))\varphi(M)\log X\frac{\log_2 X\log_4 X}{{(\log_3 X)}^2}.
  \end{aligned}
\end{gather}
Thus our task is to show that $\mathfrak{P}^{(M)}_x$ can sieve out $[1,U]$
while taking $C_U$ arbitrarily large in~\eqref{eq:U}.

We take $a_p \equiv 0 \pmod{p}$ for primes $p \in \mathfrak{P}^{(M)}$, $y<p\leq
z$, where
\begin{equation}
  z=\frac{x}{\log\log x}.
\end{equation}
The set that remains after this sieving is
\begin{equation}
  \begin{split}
    \{m \leq U : \text{$m$ is $y$-smooth}\} &\cup \{mp \leq U: \text{$p>z$, $m$
    is $y$-smooth}\}\\
    &\cup\bigcup_{\substack{p\mid M\\y<p\leq z}}\{n\leq U:p\mid n\}.
  \end{split}
\end{equation}
Denote the last union over $p\mid M$ by $E_1$. Then
\begin{gather}
  \begin{aligned}\label{eq:E1}
    \abs{E_1}&\leq\frac{U}{y}\omega(M)\\
    &\ll\frac{U\log M}{y\log_2 M}\\
    &\ll\frac{x\log y\log x}{y{(\log \log x)}^2}=o\bigl(\frac{x}{\log x}\bigr).
  \end{aligned}
\end{gather}
For the second sieving we use residue classes $a_p
\equiv 1 \pmod{p}$ for all $p \in
\mathfrak{P}^{(M)}_y$.  Then what remains is
\begin{equation}
  \begin{split}
    \{m \leq U : &\text{ $m$ is $y$-smooth, $\left(m-1,P_M(y)\right)$}\}\\
    &\cup \{mp \leq U: \text{$p>z$, $m$ is $y$-smooth, $\left( mp-1,P_M(y)
    \right)=1$}\}\\
    &\cup E_2\\
    &=\mathcal{R}_0^{(M)} \cup \mathcal{R}^{(M)} \cup E_2,
  \end{split}
\end{equation}
where $E_2$ is the result of sieving $E_1$, so $\lvert E_2 \rvert \leq \lvert
E_1 \rvert$.

We split $\mathcal{R}^{(M)}$ according to the integer $m$, and write
\begin{equation}
  \mathcal{R}_m^{(M)}=\{ z <p \leq U/m:\left(mp-1,P_M(y)\right)=1\}.
\end{equation}
Note that if both $M$ and $m$ are odd, then this set is vacuous. So we posit
the following restriction.
\begin{equation}\label{eq:assum2}
  2\nmid M\Rightarrow 2\mid m.
\end{equation}
In the sequel, $m$ will be understood to satisfy~\eqref{eq:assum2}. We have the
following estimate on the size of $\mathcal{R}_m^{(M)}$.
\begin{lem}\label{lem:Rm}
  Uniformly for $z+z/\log x\leq V\leq x{(\log x)}^2$, $M\leq\kappa x^{1/5}$ and
  $m\leq x$ satisfying~\eqref{eq:assum2}, there holds
  \begin{multline}\label{eq:fundlem}
    \#\{z<p\leq V:(mp-1,P_M(y))=1\}\\
    =\frac{V-z}{\log x}\biggl(
      \prod_{\substack{p\leq y\\p\nmid M\\p\nmid m}}
      \frac{p-2}{p-1}
    \biggr)
    \biggl(1+O(\exp(-{(\log x)}^{1/2}))\biggr).
  \end{multline}
  In particular, uniformly for $m\leq U(1-1/\log x)/z$,
  \begin{equation}\label{eq:Rm}
    \abs{\mathcal{R}_m^{(M)}}=\frac{2e^{-\gamma}U(1+o(1))}{m(\log x)(\log y)}
    \frac{M}{\varphi(M)}
    \biggl(\prod_{\substack{p>2\\p\nmid M}}\frac{p(p-2)}{{(p-1)}^2}\biggr)
    \biggl(\prod_{\substack{p>2\\p\nmid M\\p\mid m}}\frac{p-1}{p-2}\biggr)
  \end{equation}
\end{lem}
\begin{proof}
  This is almost identical to Lemma 3 of~\cite{Maynard2016Large}, the only
  difference being that in our case the primes which divide $M$ are excluded
  from the sieving process in the application of the fundamental lemma,
  effecting the constraint $p\nmid M$ in~\eqref{eq:fundlem}. Note
  that~\eqref{eq:assum2} ensures that $p=2$ does not occur in the product.
\end{proof}
\begin{lem}\label{lem:RmSums}
  For any $K\geq2$, we have
  \begin{equation}
    \sum_{U/(zK)\leq m<U/z}\abs{\mathcal{R}_m^{(M)}}
    \ll\frac{UM\log K}{(\log x)(\log y)\varphi(M)}.
  \end{equation}
  In particular,
  \begin{equation}
    \sum_{U/(z{(\log_2 x)}^2)\leq m<U/z}\abs{\mathcal{R}_m^{(M)}}
    =o\biggl(\frac{C_{U}x}{\log x}\biggr),
    \qquad\sum_{1\leq m<U/(z{(\log_2 x)}^2)}\abs{\mathcal{R}_m^{(M)}}
    =O\biggl(\frac{C_{U}x}{\log x}\biggr).
  \end{equation}
\end{lem}
\begin{proof}
  Put $w_1=U/(zK)$ and $w_2=U(1-1/\log x)/x$. For $m\geq w_2$, we use the
  trivial bound $\abs{\mathcal{R}_m^{(M)}}\ll U/(m\log x)$ to see that the
  contribution from $w_2\leq U/z$ is $O(U/{(\log x)}^2)$.

  We regroup the factors in~\eqref{eq:Rm} to separate the effect of $M$.
  \begin{equation}
    \biggl(\prod_{\substack{p>2\\p\nmid M}}\frac{p(p-2)}{{(p-1)}^2}\biggr)
    \biggl(\prod_{\substack{p>2\\p\nmid M\\p\mid m}}\frac{p-1}{p-2}\biggr)
    \leq\biggl(\prod_{p\mid M}\frac{{(p-1)}^2}{p(p-2)}\biggr)
    \biggl(\prod_{p>2}\frac{p(p-2)}{{(p-1)}^2}\biggr)
    \biggl(\prod_{\substack{p>2\\p\mid m}}\frac{p-1}{p-2}\biggr).\\
  \end{equation}
  The first product on the right hand side can be estimated as
  \begin{gather}
    \begin{aligned}
      \prod_{
        \substack{p\mid M\\p>2}
      }\biggl(1+\frac{1}{p(p-2)}\biggr)
      \leq\prod_{p\mid M}\biggl(1+\frac{3}{p^2}\biggr)\leq{\zeta(2)}^3,
    \end{aligned}
  \end{gather}
  whence we have
  \begin{equation}
    \abs{\mathcal{R}_m^{(M)}}\ll\frac{U(1+o(1))}{m(\log x)(\log y)}
    \frac{M}{\varphi(M)}
    \biggl(\prod_{\substack{p>2\\p\mid m}}\frac{p-1}{p-2}\biggr).
  \end{equation}
  Now for $w_1\leq m<w_2$, we use the bound
  \begin{equation}
\prod_{\substack{p>2\\p\mid m}}\frac{p-1}{p-2}
\ll\prod_{p\mid m}\frac{p+1}{p}
\leq\sum_{d\mid m}\frac{1}{d},
\end{equation}
and obtain
\begin{gather}
  \begin{aligned}
    \sum_{w_1\leq m<w_2}\abs{\mathcal{R}_m^{(M)}}
    &\ll\frac{UM}{(\log x)(\log y)\varphi(M)}
    \sum_{w_1\leq m<w_2}\frac{1}{m}
    \sum_{d\mid m}\frac{1}{d}\\
    &=\frac{UM}{(\log x)(\log y)\varphi(M)}
    \sum_{d<w_2}\frac{1}{d^2}
    \sum_{w_1/d\leq m<w_2/d}\frac{1}{m}\\
    &\ll\frac{UM(\log K+O(1))}{(\log x)(\log y)\varphi(M)},
    \end{aligned}
    \end{gather}
    and substituting the definitions of $U$ and $y$ yields the particular cases.
\end{proof}
We also have a bound for $\mathcal{R}_0^{(M)}$.
\begin{lem}\label{lem:Rnaught}
  We have
  \begin{equation}
    \abs{\mathcal{R}_0^{(M)}}\ll\frac{x}{{(\log x)}^{1+\varepsilon}}.
  \end{equation}
\end{lem}
\begin{proof}
  This is Theorem 5.3 in~\cite{MaierP1990Unusually}, again with the only
  difference being that prime divisors of $M$ are excluded from the sieving
  process in the invocation of Theorem 4.2 of~\cite{HalberstamR1974Sieve},
  again contributing a factor $\ll M/\varphi(M)$. By our restriction on the
  size of $M$, this can be absorbed in the ${(\log x)}^{-\varepsilon}$ factor.
\end{proof}
With these estimates, we will use the key proposition below to prove our main
result.
\begin{prop}\label{prop:main}
  Let $\delta>0$ be given, and $x>x_0(\delta)$ be large enough. For each
  $m<Uz^{-1}{(\log_2 x)}^{-2}$ satisfying~\eqref{eq:assum2}, let
  $\mathcal{I}_m\subseteq[x/2,x]$ be an interval of length at least
  $\delta\abs{\mathcal{R}_m^{(M)}}\log x$.  Then there exits a choice of
  residue classes $a_q\pmod{q}$ for each prime $q\in\mathcal{I}_m$ such that
  for all $p\in\mathcal{R}_m^{(M)}$ there is a $q\in\mathcal{I}_m$ such that
  $p\equiv a_q\pmod{q}$.
\end{prop}
\begin{proof}[Proof of Theorem~\ref{thm:main} assuming
  Proposition~\ref{prop:main}]
  By Lemma~\ref{lem:RmSums}, we have
  \begin{equation}
    \sum_{m<U/(z{(\log_2 x)}^2)}\delta\abs{\mathcal{R}_m^{(M)}}\log x
    \ll\delta C_{U}x.
  \end{equation}
  Thus, if $\delta$ is small enough, we can choose the
  $\mathcal{I}_m\subseteq[x/2,x]$ for even $m<Uz^{-1}{(\log_2 x)}^{-2}$ to be
  disjoint. By Proposition~\ref{prop:main}, the primes in those intervals are
  enough to sieve out all the primes in the $\mathcal{R}_m^{(M)}$.  By
  lemmas~\ref{lem:Rm},~\ref{lem:RmSums}~and~\ref{lem:Rnaught}, this shows that
  we can cover all but $o(x/\log x)$ numbers in
  $\mathcal{R}_0^{(M)}\cup\mathcal{R}^{(M)} \cup E_2$ using primes in
  $[x/2,x]$. So using one residue class each for the primes in $[z,x/2]$ is
  sufficient to cover what remains. This proves that $\mathfrak{P}_x^{(M)}$ can
  sieve out $[1,U]$.
\end{proof}
The proof of Proposition~\ref{prop:main} is probabilistic. Assume that for $q
\in \mathcal{I}_m$, we pick a residue class $a \pmod{q}$ with probability
$\mu_{m,q}(a)$. Then the probability that a given $p_0\in\mathcal{R}_m^{(M)}$
is not picked for any $q\in\mathcal{I}_m$ is
\begin{equation}
  \prod_{q\in\mathcal{I}_m}\biggl(1-\mu_{m,q}(p_0)\biggr)
  \leq\exp\biggl(-\sum_{q\in\mathcal{I}_m}\mu_{m,q}(p_0)\biggr).
\end{equation}
If we show that the sum on the right hand side can be made arbitrarily large,
then we can deduce that there's a choice of residue classes $a_q\pmod{q}$ such
that an arbitrarily small portion of the primes in $\mathcal{R}_m^{(M)}$ is
left out.

We put
\begin{equation}
  \begin{split}
    \omega_{m,q}(p)=\#\{&1\leq n\leq p:n+h_{i}q\equiv 0 \pmod{p}\\
    &\text{or }m(n+h_{i}q) \equiv 1 \pmod{p} \text{ for some $i=1,\ldots,k$}\},
  \end{split}
\end{equation}
where $\mathcal{H}={h_1,\ldots,h_k}$ with $h_i=p_{\pi(k)+i}P(w)$ is an
admissible $k$-tuple (recall that $\{h_i\}$ is called admissible if
$\abs{\{h_i\pmod{p}\}}<p$ for all primes $p$). Also let $\varphi_{m,q}$ be the
multiplicative function defined on primes by
$\varphi_{m,q}(p)=p-\omega_{m,q}(p)$. With this, we define the singular series
\begin{equation}\label{eq:ss}
  \mathfrak{S}_{m,q}^{(M)}=\prod_{\substack{p\leq y\\p\nmid M}}
  \left(1-\frac{\omega_{m,q}(p)}{p}\right)
  {\left(1-\frac{1}{p}\right)}^{-2k}
  \prod_{\substack{p\leq y\\p\mid M}}{\left(1-\frac{1}{p}\right)}^{-k}
  \prod_{\substack{p\leq w\\p\mid M}}{\left(1-\frac{1}{p}\right)}^{1-k}.
\end{equation}

We will define $\mu_{m,q}$ by
\begin{equation}\label{eq:mu}
  \mu_{m,q}(a) = \alpha_{m,q} \sum_{
    \substack{
      n \leq U/m\\
      n \equiv a \pmod{q}\\
      (n,P(w))=1\\
      (mn-1,P_M(w))=1
    }
  }
  {\Biggl(
    \sum_{
      \substack{
        d_1,\ldots,d_k\\
        d_i\mid(n+h_{i}q)
      }
    }
    \sum_{
      \substack{
        e_1,\ldots,e_k\\
        e_i\mid m(n+h_{i}q)-1\\
        (e_i,M)=1
      }
    }
    \lambda_{d_1,\ldots,d_k,e_1,\ldots,e_k}
    \Biggr)}^2,
\end{equation}
where $\alpha_{m,q}$ is a normalizing constant, $w=\log_4 x$ and the $\lambda$
are given by
\begin{equation}
  \lambda_{d_1,\ldots,d_k,e_1,\ldots,e_k}=\Bigl(
    \prod_{i=1}^k\mu(d_i)\mu(e_i)\sum_{j=1}^J\Bigl(
      \prod_{\ell=1}^{k}F_{\ell,j}\bigl(\frac{\log{d_\ell}}{\log{x}}\bigr)
      G\bigl(\frac{\log{e_\ell}}{\log{x}}\bigr)
    \Bigr)
  \Bigr)
  \label{eq:lambdadef}
\end{equation}
for some smooth nonnegative functions $F_{i,j}$,
$G:\left[0,\infty\right)\to\mathbb{R}$ which are not identically zero. 
These functions and the parameter $J$ may depend on $k$ but not on $x$ or $q$.
Thus $\abs{\lambda_{d_1,\ldots,d_k,e_1,\ldots,e_k}}\ll_k1$. Also for each
$j=1,\ldots,J$, we require
\begin{equation}
  \sup\bigl\{\sum_{i=1}^{k}u_i:F_{i,j}(u_i)\neq 0\bigr\}\leq 1/10,
\end{equation}
and restrict $G$ to be supported on $[0,1]$. Also put
\begin{equation}
  F(t_1,\ldots,t_k)=\sum_{j=1}^J\prod_{\ell=1}^{k}F_{\ell,j}^\prime(t_\ell),
\end{equation}
and assume that $F_{i,j}$ are chosen so that $F$ is symmetric.

Two things are different in~\eqref{eq:mu} compared to~\cite{Maynard2016Large}.
Firstly, we have the weaker condition $(mn-1,P_M(w))=1$ instead of
$(mn-1,P(w))=1$, reflecting the corresponding condition on the definition of
the $\mathcal{R}_m^{(M)}$. Also, we require that $(e_i,M)=1$ to simplify
certain divisibility conditions that will arise.

\section{Estimations}
To first estimate $\alpha_{m,q}$, we sum~\eqref{eq:mu}  over $a \pmod{q}$ and
rearrange sums to obtain
\begin{equation}
  \alpha_{m.q}^{-1}=
  \sum_{
    \substack{
      d_1,\ldots,d_k\\
      d_1^\prime,\ldots,d_k^\prime
    }
  }
  \sum_{
    \substack{
      e_1,\ldots,e_k\\
      e_1^\prime,\ldots,e_k^\prime
    }
  }
  \lambda_{\mathbf{d},\mathbf{e}}
  \lambda_{\mathbf{d}^\prime,\mathbf{e}^\prime}
  \sum_{
    \substack{
      n \leq U/m\\
      (n,P(w))=1\\
      (mn-1,P_M(w))=1\\
      [d_i,d_i^\prime]\mid n+h_{i}q\\
      [e_i,e_i^\prime]\mid m(n+h_{i}q)-1\\
      (e_{i}e_i^\prime,M)=1
    }
  }
  1.
\end{equation}
Here if $p\mid d_{i}d_i^\prime$, then $p\mid n+h_{i}q$, whence $p>w$ since
$P(w)\mid h_i$ and $(n,P(w))=1$. So we have $(d_{i}d_i^\prime,P(w))=1$ for all
$i$. Similarly if $p\mid e_{i}e_i^\prime$, then $p\nmid M$, and we also have
$p\mid mn - 1 + h_{i}qm$, so $p\nmid P_M(w)$ as well, whence
$(e_{i}e_i^\prime,P(w))=1$ for all $i$. Also, $d_{i}d_i^\prime$ and
$d_{j}d_j^\prime$ are relatively prime for $i\neq j$, since a common divisor
$p$ of both would have to divide $(h_i-h_j)q$, but this is absurd when $p\mid
h_i-h_j$ implies $p\leq w$, but $p\nmid P(w)$ and $p\leq x^{1/10}<q$. Along the
same lines, we see also that $e_{1}e_1^\prime,\ldots,e_{k}e_k^\prime$ are
pairwise relatively prime. Also, we see immediately that if
$p\mid(d_{i}d_i^\prime,e_{i}e_i^\prime)$ then $p\mid mq(h_j-h_i)-1$ and that
$(e_{i}e_i^\prime,M)=1$. Under these restrictions, the inner sum counts the
$n$ satisfying
\begin{gather}
  \begin{alignat}{3}
    n&\not\equiv 0 &&\pmod{p}, &&p\leq w,\\
    mn&\not\equiv 1 &&\pmod{p}, &&p\leq w,\:p\nmid M,\\
    n&\equiv -h_{i}q &&\pmod{[d_i,d_i^\prime]}, &&\forall i,\\
    n&\equiv \overline{m}-h_{i}q &&\pmod{[e_i,e_i^\prime]},\qquad &&\forall i.
  \end{alignat}
\end{gather}
By the Chinese Remainder Theorem, the number of such $n$ in a block of
$P(w)[\mathbf{d},\mathbf{d}^\prime,\mathbf{e},\mathbf{e}^\prime]$ integers is
$\varphi_{m,q}(P_M(w))\varphi(\frac{P(w)}{P_M(w)})$. So we have
\begin{multline}
  \alpha_{m.q}^{-1}=
  \sideset{}{^\prime}\sum_{
    \substack{
      d_1,\ldots,d_k\\
      d_1^\prime,\ldots,d_k^\prime
    }
  }
  \sideset{}{^\prime}\sum_{
    \substack{
      e_1,\ldots,e_k\\
      e_1^\prime,\ldots,e_k^\prime
    }
  }
  \lambda_{\mathbf{d},\mathbf{e}}
  \lambda_{\mathbf{d}^\prime,\mathbf{e}^\prime}\\
  \times\left(
    \frac{U\varphi_{m,q}(P_M(w))\varphi(\frac{P(w)}{P_M(w)})}
    {mP(w)[\mathbf{d},\mathbf{d}^\prime,\mathbf{e},
    \mathbf{e}^\prime]}
    + O\left(
    \varphi_{m,q}(P_M(w))\varphi(\tfrac{P(w)}{P_M(w)})
    \right)
  \right),
\end{multline}
where $\sum^\prime$ denotes the sums with the aforementioned divisibility
conditions. Note that in our case the only extra constraint compared to the
original case is $(e_{i}e_i^\prime,M)=1$.

Using the fact that $\abs{\lambda_{\mathbf{d},\mathbf{e}}}\ll_k1$, and
recalling the support conditions $\prod_{i}d_i<x^{1/10}$ and
$\prod_{i}e_i,y^k\ll x^\varepsilon$, together with the fact that
$\varphi_{m,q}(P_M(w))\varphi(\tfrac{P(w)}{P_M(w)})\leq P(w)\ll\log_3 x$, we
see that the contribution of the error term is at most $\ll x^{1/2}$.

We expand the $\lambda$ using~\eqref{eq:lambdadef}, so that we are left to
evaluate
\begin{equation}\label{eq:quadsum}
  \sum_{j=1}^J\sum_{j^\prime=1}^J
  \sideset{}{^\prime}\sum_{\substack{d_1,\ldots,d_k\\
                                     d^\prime_1,\ldots,d^\prime_k}}
  \sideset{}{^\prime}\sum_{\substack{e_1,\ldots,e_k\\
                                     e^\prime_1,\ldots,e^\prime_k}}
  \frac{
    \prod_{\ell=1}^k\mu(d_\ell)\mu(d^\prime_\ell)\mu(e_\ell)\mu(e^\prime_\ell)
    H_{\ell,j,j^\prime}(d_\ell,d^\prime_\ell,d_\ell,e^\prime_\ell)
  }
  {
    [\mathbf{d},\mathbf{d}^\prime,\mathbf{e},\mathbf{e}^\prime]
  },
\end{equation}
where
\begin{equation}
  H_{\ell,j,j^\prime}(d_\ell,d^\prime_\ell,d_\ell,e^\prime_\ell)
  =F_{\ell,j}\bigl(\frac{\log d_\ell}{\log x}\bigr)
  F_{\ell,j^\prime}\bigl(\frac{\log d^\prime_\ell}{\log x}\bigr)
  G\bigl(\frac{\log e_\ell}{\log y}\bigr)
  G\bigl(\frac{\log e^\prime_\ell}{\log y}\bigr).
\end{equation}
The functions $e^{t}F_{\ell,j}(t)$ can be extended to smooth compactly
supported functions on $\mathbb{R}$, so has a Fourier expansion
$e^{t}F_{\ell,j}(t)=\int_\mathbb{R}e^{-it\xi}f_{\ell,j}(\xi)d\xi$, with
$f_{\ell,j}(\xi)\ll_{k,A}{(1+\abs{\xi})}^{-A}$ rapidly decreasing. Thus
\begin{equation}
  F_{\ell,j}\biggl(\frac{\log d_\ell}{\log x}\biggr)
  =\int_\mathbb{R}
  \frac{f_{\ell,j}(\xi_\ell)}{d_\ell^{(1+i\xi_\ell)/\log x}}d\xi_\ell,
\end{equation}
and similarly for $G$. So we can rewrite the inner two sums
in~\eqref{eq:quadsum} as
\begin{equation}
  \begin{split}
    \int_\mathbb{R}\cdots\int_\mathbb{R}
    \Biggl(
      \sideset{}{^\prime}\sum_{\substack{d_1,\ldots,d_k\\
                                         d^\prime_1,\ldots,d^\prime_k}}
      &\sideset{}{^\prime}\sum_{\substack{e_1,\ldots,e_k\\
                                          e^\prime_1,\ldots,e^\prime_k}}
      \frac{1}{[\mathbf{d},\mathbf{d}^\prime,\mathbf{e},\mathbf{e}^\prime]}
      \prod_{\ell=1}^k\frac{
        \mu(d_\ell)\mu(d^\prime_\ell)\mu(e_\ell)\mu(e^\prime_\ell)
      }{
        d_\ell^{\frac{1+i\xi_\ell}{\log x}}
        {(d^\prime_\ell)}^{\frac{1+i\xi^\prime_\ell}{\log x}}
        e_\ell^{\frac{1+i\tau_\ell}{\log y}}
        {(e^\prime_\ell)}^{\frac{1+i\tau^\prime_\ell}{\log y}}
      }
    \Biggr)\\
    &\times
    \Bigl(
      \prod_{\ell=1}^k
      f_{\ell,j}(\xi_\ell)f_{\ell,j^\prime}(\xi^\prime_\ell)
      g(\tau_\ell)g(\tau^\prime_\ell)
      d\xi_\ell d\xi^\prime_\ell
      d\tau_\ell d\tau^\prime_\ell
    \Bigr),
    \label{eq:intsumprod}
  \end{split}
\end{equation}
and in turn write the sum here as a product $\prod_{p}K_p$, where
\begin{equation}
  \begin{split}
    K_p&=
    \sideset{}{^\prime}\sum_{
      \substack{d_1,\ldots,d_k\\d^\prime_1,\ldots,d^\prime_k}}
    \sideset{}{^\prime}\sum_{
      \substack{e_1,\ldots,e_k\\e^\prime_1,\ldots,e^\prime_k\\
      [\mathbf{d},\mathbf{d}^\prime,\mathbf{e},\mathbf{e}^\prime]\mid p}}
    \frac{1}{[\mathbf{d},\mathbf{d}^\prime,\mathbf{e},\mathbf{e}^\prime]}
    \prod_{\ell=1}^k\frac{
      \mu(d_\ell)\mu(d^\prime_\ell)\mu(e_\ell)\mu(e^\prime_\ell)
    }{
      d_\ell^{\frac{1+i\xi_\ell}{\log x}}
      {(d^\prime_\ell)}^{\frac{1+i\xi^\prime_\ell}{\log x}}
      e_\ell^{\frac{1+i\tau_\ell}{\log y}}
      {(e^\prime_\ell)}^{\frac{1+i\tau^\prime_\ell}{\log y}}
    }\\
    &=1+O_k(p^{-1-1/\log x}),
  \end{split}
\label{eq:kaypee}
\end{equation}
so that $\prod_{p}K_p\ll{(\log x)}^{O_k(1)}$. By the rapid decrease of the
functions $f,g$ we can truncate the integrals to
$\abs{\xi_\ell},\abs{\xi_\ell^\prime},\abs{\tau_\ell},\abs{\tau_\ell^\prime}
\leq {(\log x)}^{1/2}$. We relabel $s_j=(1+i\xi_j)/\log
x,r_\ell=(1+i\tau_\ell)/\log x$, and similarly for $s_j^\prime$,
$r_\ell^\prime$.

Now for $w\leq p\leq y$ with
$p\nmid M\prod_{h,h^\prime\in\mathcal{H}} (mq(h-h^\prime)-1)$, or for $p>y$, we
have
\begin{equation}\label{eq:fullfactor}
  K_p=\Bigl(1+O_k\Bigl(\frac{1}{p^2}\Bigr)\Bigl)\prod_{\ell=1}^k
    \frac{
      \bigl(1-p^{-1-s_\ell}\bigr)
      \bigl(1-p^{-1-s_\ell^\prime}\bigr)
      \bigl(1-p^{-1-r_\ell}\bigr)
      \bigl(1-p^{-1-r_\ell^\prime}\bigr)
    }{
      \bigl(1-p^{-1-s_\ell-s_\ell^\prime}\bigr)
      \bigl(1-p^{-1-r_\ell-r_\ell^\prime}\bigr)
    }
\end{equation}

If $w\leq p\leq y$ with $p\mid \prod_{h,h^\prime\in\mathcal{H}}
(mq(h-h^\prime)-1)$ and $p\nmid M$, we will have an extra factor corresponding
to products of $d_j,e_\ell$ if $p\mid mq(h_\ell-h_j)-1$.
\begin{equation}
  \begin{split}
    (1+O_k\Bigl(\frac{1}{p^2}\Bigr)\Bigl)\prod_{
      j,\ell:p\mid mq(h_\ell-h_j)-1}
      \Bigl(
        1+\frac{1}{p}\sum_{\substack{
            \mathcal{T}\subseteq\{s_j,s_j^\prime,r_\ell,r_\ell^\prime\}\\
            \mathcal{T}\cap\{s_j,s_j^\prime\}\neq\varnothing\\
            \mathcal{T}\cap\{r_\ell,r_\ell^\prime\}\neq\varnothing\\
          }
        }
        {(-1)}^{\#\mathcal{T}}p^{-\sum_{t\in\mathcal{T}}t}
      \Bigr)\\
      =\Bigl(1+\frac{\#\{j,\ell:p\mid mq(h_\ell-h_j)-1\}}{p}\Bigr)
      \Bigl(1+O_k\Bigl(
      \frac{1}{p^2}+\frac{\log p\sqrt{\log x}}{p\log y}\Bigr)\Bigr),
  \end{split}
\end{equation}
where we used the fact that $p^{-\sum t}=1+O((\log p){(\log x)}^{1/2}/(\log y))$
by the truncation of the variables. The first factor here simplifies to
$1-(\omega_{m,q}(p)-2k)/p$.

If $p\mid M$, then due to our constraint $(e_i{e_i}^\prime,M)=1$, we have no
contribution from the $e$'s, so such $p$ contribute a factor of
\begin{equation}
  K_p=\Bigl(1+O_k\Bigl(\frac{1}{p^2}\Bigr)\Bigl)\prod_{\ell=1}^k
    \frac{
      \bigl(1-p^{-1-s_\ell}\bigr)
      \bigl(1-p^{-1-s_\ell^\prime}\bigr)
    }{
      \bigl(1-p^{-1-s_\ell-s_\ell^\prime}\bigr).
    }
\end{equation}

Finally, we can supply the same factors as in~\eqref{eq:fullfactor} for small
primes by noting that
\begin{multline}
  \prod_{p\leq w}{\biggl(1-\frac{1}{p}\biggr)}^{-2k}\prod_{p\leq
  w}\prod_{\ell=1}^k
  \frac{
    \bigl(1-p^{-1-s_\ell}\bigr)
    \bigl(1-p^{-1-s_\ell^\prime}\bigr)
    \bigl(1-p^{-1-r_\ell}\bigr)
    \bigl(1-p^{-1-r_\ell^\prime}\bigr)
  }{
    \bigl(1-p^{-1-s_\ell-s_\ell^\prime}\bigr)
    \bigl(1-p^{-1-r_\ell-r_\ell^\prime}\bigr)
  }\\
  =(1+o_k(1)).
\end{multline}
Putting these together, we find that
\begin{equation}
  \begin{split}
    \prod_{p>w}K_p=\:&(1+o_k(1))\prod_{p\leq w}{\bigl(1-\frac{1}{p}\bigr)}^{-2k}
    \prod_{\substack{w<p\leq y\\p\nmid M}}
    \bigl(1-\frac{\omega_{m,q}(p)-2k}{p}\bigr)\\
    &\times\prod_{\ell=1}^k\frac{
      \zeta\bigl(1+\frac{2+i\xi_\ell+i\xi_\ell^\prime}{\log x}\bigr)
      \zeta\bigl(1+\frac{2+i\tau_\ell+i\tau_\ell^\prime}{\log y}\bigr)
    }{
      \zeta\bigl(1+\frac{1+i\xi_\ell}{\log x}\bigr)
      \zeta\bigl(1+\frac{1+i\xi_\ell^\prime}{\log x}\bigr)
      \zeta\bigl(1+\frac{1+i\xi_\tau}{\log y}\bigr)
      \zeta\bigl(1+\frac{1+i\xi_\tau^\prime}{\log y}\bigr)
    }\\
    &\times\prod_{\substack{w<p\leq y\\p\mid M}}\prod_{\ell=1}^k
    \frac{(1-p^{-1-r_\ell-r_\ell^\prime})}{
      (1-p^{-1-r_\ell})
      (1-p^{-1-r_\ell^\prime})
    }.
  \end{split}
\end{equation}
We see that the last product is
\begin{equation}
  \times\prod_{\substack{w<p\leq y\\p\mid M}}\prod_{\ell=1}^k
  \frac{(1-p^{-1-r_\ell-r_\ell^\prime})(1-p^{-1})}{
    (1-p^{-1-r_\ell})
    (1-p^{-1-r_\ell^\prime})
  }
  \prod_{\substack{w<p\leq y\\p\mid M}}{\left(1-\frac{1}{p}\right)}^{-k},
\end{equation}
and the double product can be written as
\begin{equation}
  \exp\Biggl(\sum_{\substack{w<p\leq y\\p\mid M}}
    \sum_{\ell=1}^k
    p^{-1}(
      -p^{-r_\ell-r_\ell^\prime}
      +p^{-r_\ell}
      +p^{-r_\ell^\prime}
      -1+O(p^{-1})
    )
  \Biggr).
\end{equation}
Since $\abs{r_j},\abs{r_j^\prime}<{(\log x)}^{1/2}/\log y$, we have
\begin{equation}
  p^{-r_\ell-r_\ell^\prime}, p^{-r_\ell}, p^{-r_\ell^\prime}
  =1+O({(\log p)}^2 {(\log x)}^{1/2}/\log y),
\end{equation}
so the double sum is
\begin{equation}
  O_k\biggl(
    \frac{
      {(\log x)}^{1/2}
    }{\log y
    }
    \sum_{\substack{w<p\leq y\\p\mid M}}
    \frac{{(\log p)}^2}{p}
    +o_k(1)
  \biggr).
\end{equation}
By the assumption~\eqref{eq:assum}, we see that this is
$o_k(1)$. Integrating the zeta factors proceeds identically as
in~\cite{Maynard2016Large}, so putting everything together and noting that
\begin{multline}
  \varphi_{m,q}(P_M(w))\varphi(\tfrac{P(w)}{P_M(w)})
  \prod_{p\leq w}{\bigl(1-\frac{1}{p}\bigr)}^{-2k}\\
  \times\prod_{\substack{w<p\leq y\\p\nmid M}}
  \bigl(1-\frac{\omega_{m,q}(p)-2k}{p}\bigr)
  \prod_{\substack{w<p\leq y\\p\mid M}}{\left(1-\frac{1}{p}\right)}^{-k}
  =\mathfrak{S}_{m,q}^{(M)},
\end{multline}
we obtain
\begin{lem}
  We have
  \begin{equation}\label{eq:alphainverse}
    \alpha_{m,q}^{-1}=(1+o(1))\frac{
      U\mathfrak{S}_{m,q}^{(M)}
    }{
      m{(\log x)}^k{(\log y)}^k
    }
    I_k^{(1)}(F)I_k^{(2)}(G),
  \end{equation}
  where $\mathfrak{S}_{m,q}^{(M)}$ is given by~\eqref{eq:ss}, and
  \begin{gather}
    \begin{aligned}
      I_k^{(1)}(F)&=\idotsint_{t_1,\ldots,t_k\geq0}
      {F(t_1,\ldots,t_k)}^2dt_1,\ldots,dt_k,\\
      I_k^{(2)}(G)&={\biggl(\int_0^\infty G^\prime(t)dt\biggr)}^k.
    \end{aligned}
  \end{gather}
\end{lem}

Now we can consider the sum
\begin{multline}\label{eq:lem2sum}
  \sum_{q\in\mathcal{I}_m\text{prime}}\mu_{m,q}(p_0)\\
  =\sum_{q\in\mathcal{I}_m\:\text{prime}}\alpha_{m,q} \sum_{
    \substack{
      n \leq U/m\\
      n \equiv p_0 \pmod{q}\\
      (n,P(w))=1\\
      (mn-1,P_M(w))=1
    }
  }
  {\Biggl(
      \sum_{
        \substack{
          d_1,\ldots,d_k\\
          d_i\mid(n+h_{i}q)
        }
      }
      \sum_{
        \substack{
          e_1,\ldots,e_k\\
          e_i\mid m(n+h_{i}q)-1\\
          (e_i,M)=1
        }
      }
      \lambda_{d_1,\ldots,d_k,e_1,\ldots,e_k}
  \Biggr)}^2,
\end{multline}
where $p_0\in\mathcal{R}_m^{(M)}$. We remark that even though our sifting
primes must not divide $M$, since $\mathcal{I}_m\subseteq[x/2,x]$, the $q$
under consideration are larger than $M$, so we needn't impose $q\nmid M$
explicitly.

We minorize this sum by dropping all terms except when $n=p_0-hq$ for some
$h\in\mathcal{H}$ which are clearly in the sum. Thus
\begin{multline}
  \sum_{q\in\mathcal{I}_m\text{prime}}\mu_{m,q}(p_0)\\
  \geq\sum_{h\in\mathcal{H}}
  \sum_{q\in\mathcal{I}_m\:\text{prime}}\alpha_{m,q}
  {\Biggl(
      \sum_{
        \substack{
          d_1,\ldots,d_k\\
          d_i\mid(p_0+(h_{i}-h)q)
        }
      }
      \sum_{
        \substack{
          e_1,\ldots,e_k\\
          e_i\mid m(p_0+(h_{i}-h)q)-1\\
          (e_i,M)=1
        }
      }
      \lambda_{d_1,\ldots,d_k,e_1,\ldots,e_k}
  \Biggr)}^2.
\end{multline}
In turn, we split the sum over $q$ into residue classes modulo $P(w)$ and
obtain
\begin{multline}\label{eq:mulower}
  \sum_{h\in\mathcal{H}}
  \sum_{\substack{
      w_0\pmod{P(w)}\\
      (w_0,P(w))=1
    }
  }\\
  \times\sum_{\substack{
      q\in\mathcal{I}_m\:\text{prime}\\
      q\equiv w_0\pmod{P(w)}
    }
  }\alpha_{m,q}
  {\Biggl(
      \sum_{
        \substack{
          d_1,\ldots,d_k\\
          d_i\mid(p_0+(h_{i}-h)q)
        }
      }
      \sum_{
        \substack{
          e_1,\ldots,e_k\\
          e_i\mid m(p_0+(h_{i}-h)q)-1\\
          (e_i,M)=1
        }
      }
      \lambda_{d_1,\ldots,d_k,e_1,\ldots,e_k}
  \Biggr)}^2.
\end{multline}
We now replace $\alpha_{m,q}$ by an expression with less dependence on $q$.
We note that for $p\leq w$, we have $\omega_{m,q}(p)=1$ or $2$ according as
$p\mid m$ or not, and for $w<p\leq y$, we have $\omega_{m,q}(p)=2k$ if
\begin{equation}
  p\nmid\prod_{h,h^\prime\in\mathcal{H}}(mq(h-h^\prime)-1).
\end{equation}
So,
\begin{align}\label{eq:sing}
  {\left(\mathfrak{S}_{m,q}^{(M)}\right)}^{-1}&=
  \prod_{\substack{p\leq y\\p\nmid M}}
  {\left(1-\frac{\omega_{m,q}(p)}{p}\right)}^{-1}
  {\left(1-\frac{1}{p}\right)}^{2k}
  \prod_{\substack{p\leq y\\p\mid M}}{\left(1-\frac{1}{p}\right)}^{k}
  \prod_{\substack{p\leq w\\p\mid M}}{\left(1-\frac{1}{p}\right)}^{k-1}\\
  &\geq(1+o_k(1))\mathfrak{S}_m^{(M)}
  \prod_{\substack{
      w<p\leq y\\
      p\nmid M\\
      p\mid\prod\limits_{h^\prime,h^{\prime\prime}}
      (mq(h^\prime-h^{\prime\prime})-1)
    }
  }{\left(1-\frac{1}{p}\right)}^{2k},
\end{align}
where
\begin{equation}
  \mathfrak{S}_m^{(M)}=2^{-(2k-1)}
  {\left(\frac{\varphi(M)}{M}\right)}^k
  \prod_{\substack{2<p\\p\nmid M\\p\mid m}}\left(\frac{p-2}{p-1}\right)
  \prod_{\substack{2<p\leq w\\p\nmid M}}{\left(1-\frac{1}{p}\right)}^{2k}
  {\left(1-\frac{2}{p}\right)}^{-1}.
\end{equation}
We restrict the primes occuring in the last product in~\eqref{eq:sing} to be
less than $z_0=\log x\log_3x/\log_2x$ at a cost of $(1+o_k(1))$ and expand the
product to obtain
\begin{equation}
  {\left(\mathfrak{S}_{m,q}^{(M)}\right)}^{-1}\geq(1+o_k(1))
  \mathfrak{S}_m^{(M)}\sum_{\substack{
      a_{1,2},\ldots,a_{k,k-1}\mid P_M(z_0)/P_M(w)\\
      a_{i,j}\mid mq(h_i-h_j)-1
    }
  }
  \frac{{(-2k)}^{\omega([\mathbf{a}])}}{[\mathbf{a}]},
\end{equation}
where $[\mathbf{a}]=[a_{1,2},\ldots,a_{k,k-1}]$. Substituting this and
\eqref{eq:alphainverse} in~\eqref{eq:mulower}, we obtain
\begin{equation}\label{eq:bigsum2}
  \begin{split}
  \sum_{q\in\mathcal{I}_m\text{prime}}&\mu_{m,q}(p_0)\geq
  \frac{
    (1+o(1))
    \mathfrak{S}_m^{(M)}
    m{(\log x)}^k{(\log y)}^k
  }{
    UI_k^{(1)}(F)I_k^{(2)}(G)
  }\\
  &\times
  \sum_{h\in\mathcal{H}}
  \sum_{\substack{
      w_0\pmod{P(w)}\\
      (w_0,P(w))=1
    }
  }
  \sum_{a_{1,2},\ldots,a_{k,k-1}\mid P_M(z_0)/P_M(w)}
  \frac{{(-2k)}^{\omega([\mathbf{a}])}}{[\mathbf{a}]},
  \\
  &\times\sum_{\substack{
      q\in\mathcal{I}_m\:\text{prime}\\
      q\equiv w_0\pmod{P(w)}\\
      a_{i,j}\mid mq(h_i-h_j)-1
    }
  }
  {\Biggl(
      \sum_{
        \substack{
          d_1,\ldots,d_k\\
          d_i\mid(p_0+(h_{i}-h)q)
        }
      }
      \sum_{
        \substack{
          e_1,\ldots,e_k\\
          e_i\mid m(p_0+(h_{i}-h)q)-1\\
          (e_i,M)=1
        }
      }
      \lambda_{d_1,\ldots,d_k,e_1,\ldots,e_k}
  \Biggr)}^2.
  \end{split}
\end{equation}
We consider the sum over $q$. We suppose $h$ in the outer sum is $h_k$, without
any loss of generality. By the support of $F$ and the fact that $p_0>x$, we
must have $d_k=1$, and similarly $e_k=1$. We expand the square and rearrange
the sums to find that the sum over $q$ equals
\begin{equation}\label{eq:doublesum}
  \sum_{\substack{
      d_1,\ldots,d_k\\
      d_1^\prime,\ldots,d_k^\prime\\
      d_k=d_k^\prime=1
    }
  }
  \sum_{\substack{
      e_1,\ldots,e_k\\
      e_1^\prime,\ldots,e_k^\prime\\
      e_k=e_k^\prime=1
    }
  }
  \lambda_{\mathbf{d},\mathbf{e}}\lambda_{\mathbf{d}^\prime,\mathbf{e}^\prime}
  \sum_{\substack{
      q\in\mathcal{I}_m\:\text{prime}\\
      q\equiv w_0\pmod{P(w)}\\
      a_{i,j}\mid mq(h_i-h_j)-1\\
      [d_i,d_i^\prime]\mid p_0+(h_i-h_k)q\\
      [e_i,e_i^\prime]\mid mp_0+m(h_i-h_k)q-1\\
      ((e_i,e_i^\prime),M)=1
    }
  }
  1.
\end{equation}
We see that if $p\mid{d_i}d_i^\prime$ for some $p\leq w$, then
$[d_i,d_i^\prime]\mid p_0+(h_i-h_k)q$ would imply $p\mid p_0$, an absurdity.
Thus we have ${d_i}d_i^\prime$ relatively prime to $P(w)$ for all $i$. Also, if
$p\mid({d_i}d_i^\prime,{d_j}d_j^\prime)$, this would imply $p\mid q(h_i-h_j)$,
but $h_i-h_j$ only has prime divisors not exceeding $w$, and $p<x^{1/10}<q$, so
we also have $({d_i}d_i^\prime,{d_j}d_j^\prime)=1$ for all $i\neq j$.
Similarly, if $p\mid {e_i}e_i^\prime$ with $p\leq w$, then since $p\mid
h_i-h_k$, then necessarily $p\mid mp_0-1$, but $p_0$, being in
$\mathcal{R}_m^{(M)}$, is relatively prime to $P_M(w)$, so since by assumption
$p\nmid M$, we have $({e_i}e_i^\prime,P(w))=1$ for all $i$. Similarly it is
easy to see that $({e_i}e_i^\prime,{e_j}e_j^\prime)=1$ for $i\neq j$. Plainly
$(a_{i,j},m)=1$ for all $i\neq j$, and any prime dividing
$(a_{i,j},a_{i^\prime,j^\prime})$ would have to divide
$q(h_i+h_{j^\prime}-(h_{i^\prime}+h_j))$, but the latter has no prime divisors
in $[w,z_0]$, so the $a_{i,j}$ are also pairwise coprime.  We have the
compatibility conditions
\begin{gather}\label{eq:constraints}
  \begin{alignat}{2}
  &({d_i}d_i^\prime,{e_j}e_j^\prime)\mid mp_0(h_i-h_j)+h_k-h_i
  &\forall i,j,\\
  &({d_i}d_i^\prime,a_{j,\ell})\mid (h_j-h_\ell)mp_0+h_i-h_k&\forall i,j,\ell,\\
  &({e_i}e_i^\prime,a_{j,\ell})\mid (h_j-h_\ell)(1-p_0m)-h_i+h_k
  &\forall i,j,\ell,
\end{alignat}
\end{gather}
under which the inner sum counts primes in a single residue class modulo the
least common multiple of $d_1,d_1^\prime,e_1,e_1^\prime,
\ldots,d_k,d_k^\prime,e_k,e_k^\prime,a_{1,2}, \ldots,a_{k,k-1},P(w)$. With this
we see that the inner sum in~\eqref{eq:doublesum} is
\begin{equation}\label{eq:primecount}
  \frac{
    \sum_{q\in\mathcal{I}_m\:\text{prime}}1
  }
  {
    \varphi(P(w))\varphi([
    \mathbf{d},\mathbf{d}^\prime,\mathbf{e},\mathbf{e}^\prime,\mathbf{a}])
  }
  +E(x;P(w)[
  \mathbf{d},\mathbf{d}^\prime,\mathbf{e},\mathbf{e}^\prime,\mathbf{a}]).
\end{equation}
Summing the error terms is handled by a standard application of the
Bombieri-Vinogradov theorem. With this and the fact that the number of primes
in $\mathcal{I}_m$ is $(1+o(1))\abs{\mathcal{I}_m}/\log x$, we see
that~\eqref{eq:doublesum} simplifies to
\begin{gather}\label{eq:bigsum}
  \begin{aligned}
    &\!\!\!\!\!\!\!\!\!\!\!\frac{
      (1+o(1))\abs{\mathcal{I}_m}
    }{
      \varphi(P(w))\log x
    }
    \sideset{}{^*}\sum_{\substack{
        d_1,\ldots,d_k\\
        d_1^\prime,\ldots,d_k^\prime\\
        d_k=d_k^\prime=1
      }
    }
    \sideset{}{^*}\sum_{\substack{
        e_1,\ldots,e_k\\
        e_1^\prime,\ldots,e_k^\prime\\
        e_k=e_k^\prime=1
      }
    }
    \frac{
      \lambda_{\mathbf{d},\mathbf{e}}
      \lambda_{\mathbf{d}^\prime,\mathbf{e}^\prime}
    }{
      \varphi([
      \mathbf{d},\mathbf{d}^\prime,\mathbf{e},\mathbf{e}^\prime,\mathbf{a}
  ])
    }\\
    =&\frac{
      (1+o(1))\abs{\mathcal{I}_m}{G(0)}^2
    }{
      \varphi(P(w))\log x
    }
    \sum_{j=1}^J\sum_{j^\prime=1}^J
    F_{k,j}(0)F_{k,j^\prime}(0)
    \sideset{}{^{**}}\sum_{\substack{d_1,\ldots,d_{k-1}\\
    d^\prime_1,\ldots,d^\prime_{k-1}}}
    \sideset{}{^{**}}\sum_{\substack{e_1,\ldots,e_{k-1}\\
    e^\prime_1,\ldots,e^\prime_{k-1}}}\\
    \quad&\times\frac{
      \prod_{\ell=1}^k\mu(d_\ell)\mu(d^\prime_\ell)\mu(e_\ell)\mu(e^\prime_\ell)
      F_{\ell,j}\bigl(\frac{\log d_\ell}{\log x}\bigr)
      F_{\ell,j^\prime}\bigl(\frac{\log d_\ell^\prime}{\log x}\bigr)
      G\bigl(\frac{\log e_\ell}{\log y}\bigr)
      G\bigl(\frac{\log e_\ell^\prime}{\log y}\bigr)
    }
    {
      \varphi([
          \mathbf{d},\mathbf{d}^\prime,\mathbf{e},\mathbf{e}^\prime,\mathbf{a}
      ])
    },
  \end{aligned}
\end{gather}
where $\sum^*$ indicates the divisibility conditions together with the
condition $d_k=d_k^\prime=e_k=e_k^\prime$, and $\sum^{**}$ the same with the
latter dropped. One difference in our case is that the restrictions in the sums
include $({e_i}e_i^\prime,M)=1$, and the $a_{i,j}$ also carry the restriction
$(a_{i,j},M)=1$.

We handle the sums over the $d_i,d_i^\prime,e_i,e_i^\prime$ as before by first
factorizing into prime factors $K_p$. The presence of the Euler
$\varphi$-function in the denominator effects a factor of $(1+O(p^{-2}))$ in
each $K_p$ so the difference is negligible. Let us first suppose that all the
$a_{i,j}$ are $1$. This time the primes that contribute mixed factors involving
$p\mid(d_j,e_\ell)$ must divide
$\prod_{h,h^\prime\in\mathcal{H}}(mp_0(h-h^\prime)+h_k-h)$.  So for $w\leq
p\leq y$ with $p\mid\prod_{h,h^\prime\in\mathcal{H}} (mp_0(h-h^\prime)+h_k-h)$
and $p\nmid M$, there is the contribution from products of $d_j$, $e_\ell$ if
$p\mid(mp_0(h_j-h_\ell)+h_k-h_j)$.
\begin{equation}
  \begin{split}
    (1+O_k\Bigl(\frac{1}{p^2}\Bigr)\Bigl)\prod_{
      j,\ell:p\mid(mp_0(h_j-h_\ell)+h_k-h_j)}
      \Bigl(
        1+\sum_{\substack{
            \mathcal{T}\subseteq\{s_j,s_j^\prime,r_\ell,r_\ell^\prime\}\\
            \mathcal{T}\cap\{s_j,s_j^\prime\}\neq\varnothing\\
            \mathcal{T}\cap\{r_\ell,r_\ell^\prime\}\neq\varnothing\\
          }
        }
        {(-1)}^{\#\mathcal{T}}p^{-\sum_{t\in\mathcal{T}}t}
      \Bigr)\\
      =\Bigl(1+\frac{\#\{j,\ell:p\mid(mp_0(h_j-h_\ell)+h_k-h_j)\}}{p}\Bigr)
      \Bigl(1+O_k\Bigl(
      \frac{1}{p^2}+\frac{\log p\sqrt{\log x}}{p\log y}\Bigr)\Bigr),
  \end{split}
\end{equation}
where this time the first factor simplifies to $1-(\omega^\prime_{m,p_0,h_k}
(p)-(2k-2))/p$,
with
\begin{equation}
  \begin{split}
    \omega^\prime_{m,p_0,h}=\#\{1&\leq n\leq p:p_0+(h_i-h)n\equiv0\pmod{p}\\
    &\text{or }m(p_0+(h_i-h)n)\equiv1\pmod{p}\text{ for some $i$}\}.
  \end{split}
\end{equation}
The other ranges contribute the same factors as before, so for
$a_{1,2}=\cdots=a_{k,k-1}=1$, the contribution to~\eqref{eq:bigsum} is
\begin{gather}\label{eq:secondasymp}
  \frac{
    (1+o(1))\mathfrak{S}_{m,p_0,h_k}^{(M)}\abs{\mathcal{I}_m}{G(0)}^2
  }{
    \varphi(P(w)){(\log x)}^k{(\log y)}^{k-1}
  }
  \sum_{j=1}^J\sum_{j^\prime=1}^J
  F_{k,j}(0)F_{k,j^\prime}(0)
  {\biggl(
      \int_0^\infty{G^\prime(t)}^2dt
  \biggr)}^{k-1}\\
  \begin{aligned}\nonumber
    &\times\prod_{\ell=1}^{k-1}
    \int_0^\infty F_{\ell,j}^\prime(t)F_{\ell,j^\prime}^\prime(t)dt\\
    =&\frac{
      (1+o(1))\mathfrak{S}_{m,p_0,h_k}^{(M)}\abs{\mathcal{I}_m}
    }{
      \varphi(P(w)){(\log x)}^k{(\log y)}^{k-1}
    }
    J_k^{(1)}(F)J_k^{(2)}(G),
  \end{aligned}
\end{gather}
where
\begin{multline}
  \mathfrak{S}_{m,p_0,h}^{(M)}=\prod_{p\leq w}{\biggl(
      1-\frac{1}{p}
  \biggr)}^{-(2k-2)}
  \prod_{\substack{w<p\leq y\\p\nmid M}}{
    \biggl(1-\frac{\omega_{m,p_0,h}^\prime(p)}{p}\biggr)
    {\biggl(1-\frac{1}{p}\biggr)}^{-(2k-2)}
  }\\
  \times\prod_{\substack{w<p\leq y\\p\mid M}}{
    {\biggl(1-\frac{1}{p}\biggr)}^{-(k-1)}
  }.
\end{multline}
Now if not all $a_{i,j}=1$, the presence of the $\mathbf{a}$ in the denominator
means that for any $p\mid a_{i,j}$, we have $K_p\ll_{k}p^{-1}$. Therefore the
contribution of the terms $a_{1,2},\ldots,a_{k,k-1}\neq1,\ldots,1$ is
\begin{gather}
  \begin{aligned}
    \ll_k\sum_{a_{1,2},\ldots,a_{k,k-1}\mid P_M(z_0)/P_M(w)}
    &\frac{{(-2k)}^{\omega([\mathbf{a}])}}{[\mathbf{a}]}\\
    &\times\Biggl(\frac{
        (1+o(1))\mathfrak{S}_{m,p_0,h_k}^{(M)}\abs{\mathcal{I}_m}
      }{
        \varphi(P(w)){(\log x)}^k{(\log y)}^{k-1}
      }
      J_k^{(1)}(F)J_k^{(2)}(G)
      \prod_{p\mid\prod a_{i,j}}
    \frac{O_k(1)}{p}\Biggr)\\
    &\ll_k
    \frac{
        (1+o(1))\mathfrak{S}_{m,p_0,h_k}^{(M)}\abs{\mathcal{I}_m}
      }{
        \varphi(P(w)){(\log x)}^k{(\log y)}^{k-1}
      }
      \Biggl(\prod_{w<p\leq z_0}\biggl(1+\frac{O_k(1)}{p^2}\biggr)-1\Biggr)\\
      &=o_k\biggl(
    \frac{
        \mathfrak{S}_{m,p_0,h_k}^{(M)}\abs{\mathcal{I}_m}
      }{
        \varphi(P(w)){(\log x)}^k{(\log y)}^{k-1}
      }
    \biggr),
  \end{aligned}
\end{gather}
which will be negligible. Using~\eqref{eq:secondasymp} in~\eqref{eq:bigsum2},
we obtain
\begin{gather}
  \begin{aligned}
    \sum_{q\in\mathcal{I}_m\text{prime}}\mu_{m,q}(p_0)
    \geq&\frac{
      (1+o(1))m
      (\log y)
      \abs{\mathcal{I}_m}
      J_k^{(1)}(F)J_k^{(2)}(G)
    }{
      U\varphi(P(w))
      I_k^{(1)}(F)I_k^{(2)}(G)
    }\\
    &\times
    \sum_{h\in\mathcal{H}}\mathfrak{S}_m^{(M)}\mathfrak{S}_{m,p_0,h_k}^{(M)}
    \sum_{\substack{w_0\pmod{P(w)}\\(w_0,P(w))=1}}1.
  \end{aligned}
\end{gather}
The inner sum is clearly $\varphi(P(w))$, and we have the bound
$\omega^\prime_{m,p_0,h}(p)\leq 2k-2$ uniformly in $h$. Thus
\begin{gather}
  \begin{aligned}
    \mathfrak{S}_{m,p_0,h_k}^{(M)}\mathfrak{S}_m^{(M)}
    =&2^{-1}{\biggl(\frac{\varphi(M)}{M}\biggr)}^k
    \prod_{\substack{2<p\\p\nmid M\\p\mid m}}
    \frac{p-2}{p-1}
    \prod_{\substack{2<p\leq w\\p\mid M}}{\biggl(1-\frac{1}{p}\biggr)}^{-(2k-2)}
    \prod_{\substack{2<p\leq w\\p\nmid M}}
    \frac{
      {\bigl(1-\frac{1}{p}\bigr)}^2
    }{
      \bigl(1-\frac{2}{p}\bigr)
    }\\
    &\times
    \prod_{\substack{w<p\leq y\\p\nmid M}}
    \biggl(1-\frac{\omega^\prime_{m,p_0,h}(p)}{p}\biggr)
    {\biggl(1-\frac{1}{p}\biggr)}^{-(2k-2)}
    \prod_{\substack{w<p\leq y\\p\mid M}}
    {\biggl(1-\frac{1}{p}\biggr)}^{-(k-1)}\\
    &\gg(1+o_k(1))\frac{\varphi(M)}{M}
    \prod_{\substack{p>2\\p\nmid M}}\frac{{(p-1)}^2}{p(p-2)}
    \prod_{\substack{2<p\\p\nmid M\\p\mid m}}
    \frac{p-2}{p-1}\\
    &\gg(1+o(1))\frac{2e^{-\gamma}U}{m(\log x)(\log y)\abs{\mathcal{R}_m^{(M)}}}
  \end{aligned}
\end{gather}
by Lemma~\ref{lem:Rm}. This gives
\begin{lem}\label{lem:probsum}
  Let $m<Uz^{-1}{(\log x)}^{-2}$, and let $p_0\in\mathcal{R}_m^{(M)}$
  with $h_{k}x<p_0<U/m-h_{k}x$. Then
  \begin{equation}
    \sum_{q\in\mathcal{I}_m\text{prime}}\mu_{m,q}(p_0)
    \gg\frac{
      (1+o(1))
      k\abs{\mathcal{I}_m}
    }{
      (\log x)\abs{\mathcal{R}_m^{(M)}}
    }\\
    \cdot\frac{J_k^{(1)}(F)J_k^{(2)}(G)}{I_k^{(1)}(F)I_k^{(2)}(G)}
  \end{equation}
  where
  \begin{gather}
    \begin{aligned}
      J_k^{(1)}(F)&=\idotsint_{t_1,\ldots,t_{k-1}\geq0}
      {\biggl(\int_{t_k\geq0}F(t_1,\ldots,t_k)dt_k\biggr)}^2
      dt_1,\ldots,dt_{k-1},\\
      J_k^{(2)}(G)&={G(0)}^2{\biggl(\int_0^\infty G^\prime(t)dt\biggr)}^{k-1}.
    \end{aligned}
  \end{gather}
\end{lem}
We reproduce here Lemma 8 of~\cite{Maynard2016Large}:
\begin{lem}\label{lem:k}
  There exists a choice of smooth functions $F$, $G$ such that
  \begin{equation}
    \frac{kJ_k^{(1)}(F)J_k^{(2)}(G)}{I_k^{(1)}(F)I_k^{(2)}(G)}\gg\log k.
  \end{equation}
\end{lem}
This gives us all the ingredients required to prove the main proposition.
\begin{proof}[Proof of Proposition~\ref{prop:main}] Suppose
  $\mathcal{I}_m\subseteq[x/2,x]$ is an interval of length at least
  $\delta\abs{\mathcal{R}_m^{(M)}}\log x$. By lemmas~\ref{lem:probsum}
  and~\ref{lem:k}, we see that any prime $p_0\in\mathcal{R}_m^{(M)}$ with
  $h_{k}x<p_0<U/m-h_{k}x$ has the expected number of times it is chosen
  $\sum_q\mu_{m,q}(p_0)\gg\delta\log k$. By~\eqref{eq:fundlem} we see that when
  $m<Uz^{-1}{(\log_2 x)}^{-2}$ we have $\abs{\mathcal{R}_m^{(M)}}\gg x(\log_2
  x)/\log x$, so the number of primes in $\mathcal{R}_m^{(M)}$ which are not
  considered is $o_k(\abs{\mathcal{R}_m^{(M)}})$. Given $\varepsilon$ and
  $\delta$, we choose $k$ sufficiently large so that
  $\sum_q\mu_{m,q}(p_0)>-\log\varepsilon$. Then the probability that $p_0$ is
  not in any of the chosen residue classes is
\begin{equation}
  \prod_{q\in\mathcal{I}_m\:\text{prime}}(1-\mu_{m,q}(p_0))\leq
  \exp\biggl(-\sum_{q\in\mathcal{I}_m\:\text{prime}}\mu_{m,q}(p_0)\biggr)
  \leq\varepsilon.
\end{equation}
So the expected number of primes in $\mathcal{R}_m^{(M)}$ which are not chosen
is $\varepsilon\abs{\mathcal{R}_m^{(M)}}$. Then there is at least one choice of
residue classes which leaves out at most $\varepsilon\abs{\mathcal{R}_m^{(M)}}$
primes. If we now append to $\mathcal{I}_m$ an interval of length
$2\varepsilon\abs{\mathcal{R}_m^{(M)}}\log x$, for each prime in the appended
integral we can use the residue class of one of the
$\varepsilon\abs{\mathcal{R}_m^{(M)}}$ primes that were left out. This shows
that we can cover $\abs{\mathcal{R}_m^{(M)}}$ using primes from the interval
$\mathcal{I}_m\subseteq [x/2,x]$ which has length
$(\delta+2\varepsilon)\abs{\mathcal{R}_m^{(M)}}\log x$. Since $\delta$ and
$\varepsilon$ were arbitrary, we obtain the result by relabeling.
\end{proof}
\section{Discussion}
The constraint $M\leq\kappa{(\log X)}^{1/5}$ seems rather severe. Zaccagnini
adopts the convention that $G(X;M,a)=X$ if there's no prime below $X$ in the
progression. We adopted our restriction to avoid having to engage with such
degenerate cases. We would have to make similarly severe restrictions in any
case.  Following~\eqref{eq:lem2sum}, had the magnitude of $M$ not been already
restricted, we would have to explicitly impose $(q,M)=1$. In turn, $M$ would
have to be a factor in the moduli when we count the primes
in~\eqref{eq:primecount}, which would necessitate a similarly severe
restriction (as well as further exclusions depending on the divisibility of $M$
by exceptional moduli) anyway for the Bombieri-Vinogradov theorem to be
applicable. With little gain, we opted to make the restriction
upfront and to at least ensure that the gaps we detect are nontrivial,
in the sense that they are blocks of composites that indeed fall between two
primes.


\begin{bibdiv}
  \begin{biblist}
    \bib{Erdos1940Difference}{article}{
      title={The difference of consecutive primes},
      author={Erd\H {o}s, P.},
      journal={Duke Math.\ J.},
      volume={6},
      pages={438--441},
      date={1940},
    }

    \bib{FordGKMT2018Long}{article}{
      title={Long gaps between primes},
      author={Ford, K.},
      author={Green, B.},
      author={Konyagin, S.},
      author={Maynard, J.},
      author={Tao, T.},
      journal={J.\ Amer.\ Math.\ Soc.},
      volume={31},
      number={1},
      pages={65--105},
      year={2018},
    }

    \bib{FordGKT2016Large}{article}{
      title={Large gaps between consecutive prime numbers},
      author={Ford, K.},
      author={Green, B.},
      author={Konyagin, S.},
      author={Tao, T.},
      journal={Ann.\ of Math.},
      pages={935--974},
      year={2016},
    }

    \bib{HalberstamR1974Sieve}{book}{
      title={Sieve methods},
      author={Halberstam, H.},
      author={Richert, H.\ E.},
      year={1974},
      publisher={Academic Press},
    }

    \bib{MaierP1990Unusually}{article}{
      title={Unusually large gaps between consecutive primes},
      author={Maier, H.},
      author={Pomerance, C.},
      journal={Trans.\ Amer.\ Math.\ Soc.},
      volume={322},
      number={1},
      pages={201--237},
      year={1990},
    }

    \bib{Maynard2016Large}{article}{
      title={Large gaps between primes},
      author={Maynard, J.},
      journal={Ann.\ of Math.},
      volume={183},
      number={3},
      pages={915--933},
      year={2016},
    }

    \bib{Pintz1997Very}{article}{
      title={Very large gaps between consecutive primes},
      author={Pintz, J.},
      journal={J. Number Theory},
      volume={63},
      number={2},
      pages={286--301},
      year={1997},
    }

    \bib{Rankin1938Difference}{article}{
      title={The difference between consecutive prime numbers},
      author={Rankin, R. A.},
      journal={J.\ London Math.\ Soc.},
      number={13},
      pages={242--247},
      date={1938},
    }

    \bib{Rankin1963Difference}{article}{
      title={The difference between consecutive prime numbers V},
      author={Rankin, R. A.},
      journal={Proc.\ Edinburgh Math.\ Soc.},
      volume={13},
      number={4},
      pages={331--332},
      year={1963},
    }

    \bib{Schonhage1963Bemerkung}{article}{
      title={Eine Bemerkung zur Konstruktion grosser Primzahll{\"u}cken},
      author={Sch{\"o}nhage, Arnold},
      journal={Arch.\ Math.},
      volume={14},
      number={1},
      pages={29--30},
      year={1963},
      publisher={Springer},
    }

    \bib{Westzynthius1931Uber}{article}{
      title={{\"U}ber die Verteilung der Zahlen, die zu den n ersten Primzahlen teilerfremd sind},
      author={Westzynthius, E.},
      journal={Comm. Phys. Math., Soc. Sci. Fennica},
      volume={5},
      number={25},
      pages={1--37},
      year={1931},
    }

    \bib{Xylouris2011Uber}{thesis}{
      title={{\"U}ber die Nullstellen der Dirichletschen L-Funktionen und die kleinste Primzahl in einer arithmetischen Progression},
      author={Xylouris, T.},
      year={2011},
    }

    \bib{Zaccagnini1992Note}{article}{
      title={A note on large gaps between consecutive primes in arithmetic progressions},
      author={Zaccagnini, A.},
      journal={J. Number Theory},
      volume={42},
      number={1},
      pages={100--102},
      year={1992},
    }
  \end{biblist}
\end{bibdiv}
\end{document}